\documentclass{amsart}
\usepackage{amssymb}
\usepackage{amsfonts}

\newtheorem{theorem}{Theorem}
\theoremstyle{plain}

\newtheorem{definition}{Definition}

\newtheorem{lemma}{Lemma}

\numberwithin{equation}{section}
\input{tcilatex}

\begin{document}
\title[Some Hadamard-Type Inequalities]{On Some Hadamard-Type Inequalities
for Convex Functions On A Rectanguler Box }
\author{$^{\bigstar }$M. Emin \"{O}ZDEM\.{I}R}
\address{$^{\bigstar }$Atat\"{u}rk University, K. K. Education Faculty,
Department of Mathematics, 25240, Kampus, Erzurum, Turkey}
\email{emos@atauni.edu.tr}
\author{$^{\blacklozenge ,\clubsuit }$Ahmet Ocak AKDEM\.{I}R}
\address{$^{\blacklozenge }$A\u{g}r\i\ \.{I}brahim \c{C}e\c{c}en University,
Faculty of Science and Arts, Department of Mathematics, 04100, A\u{g}r\i ,
Turkey}
\email{ahmetakdemir@agri.edu.tr}
\date{September 2010}
\keywords{Convex, $\ $co-ordinated convex, Co-ordinates, Hadamard's
inequality\\
$^{\clubsuit }$Corresponding author}

\begin{abstract}
In this paper some Hadamard-type inequalities for convex functions of $3-$%
variables on a rectanguler box are given. We also define a mapping related
to convex functions on a rectanguler box.
\end{abstract}

\maketitle

\section{Introduction}

Let $f:I\subseteq 
\mathbb{R}
\rightarrow 
\mathbb{R}
$ be a convex function defined on the interval $I$ of real numbers and $a<b.$
The following inequality;%
\begin{equation*}
f\left( \frac{a+b}{2}\right) \leq \frac{1}{b-a}\dint\limits_{a}^{b}f(x)dx%
\leq \frac{f(a)+f(b)}{2}
\end{equation*}%
is known in the literature as Hadamard's inequality for convex mappings. The
history of\ this famous integral inequality begins with the papers of Ch.
Hermite \cite{HER} and J. Hadamard \cite{HA}. In recent years there have
been many extensions, generalizations, refinements and similar results of
Hadamard's inequality, we refer interest of readers to \cite{S1}-\cite{ZAB}.
On all of these, in \cite{SS}, Dragomir defined convex functions on the
co-ordinates as following;

Let us consider the bidimensional interval $\Delta :=\left[ a,b\right]
\times \left[ c,d\right] $ in $%
\mathbb{R}
^{2}$ with $a<b$ and $c<d.$ A function $f:\Delta \rightarrow 
\mathbb{R}
$ will be called convex on the co-ordinates if the partial mappings $f_{y}:%
\left[ a,b\right] \rightarrow 
\mathbb{R}
,$ $f_{y}(u)=f(u,y)$ and $f_{x}:\left[ c,d\right] \rightarrow 
\mathbb{R}
,$ $f_{x}(v)=f(x,v)$, are convex where defined for all $x\in \left[ a,b%
\right] $ and $y\in \left[ c,d\right] .$

Recall that the mapping $f:\Delta \rightarrow 
\mathbb{R}
$ is convex in $\Delta $ if the following inequality:%
\begin{equation*}
f(\lambda x+\left( 1-\lambda \right) z,\lambda y+\left( 1-\lambda \right)
w)\leq \lambda f(x,y)+\left( 1-\lambda \right) f(z,w)
\end{equation*}%
holds for all $(x,y),$ $(z,w)\in \Delta $ and $\lambda \in \left[ 0,1\right]
.$ Therefore, in \cite{SS}, Dragomir proved the following lemma and
established following inequalities for convex functions on the co-ordinates.

\begin{lemma}
Every convex mapping $f:\Delta \rightarrow 
\mathbb{R}
$ is convex on the co-ordinates, but the converse is not generally true.
\end{lemma}

\begin{theorem}
Suppose that $f:\Delta :=\left[ a,b\right] \times \left[ c,d\right]
\rightarrow 
\mathbb{R}
$ is convex on the co-ordinates on $\Delta .$ Then one has the inequalities:%
\begin{eqnarray}
&&f\left( \frac{a+b}{2},\frac{c+d}{2}\right)  \notag \\
&\leq &\frac{1}{2}\left[ \frac{1}{b-a}\dint\limits_{a}^{b}f\left( x,\frac{c+d%
}{2}\right) dx+\frac{1}{d-c}\dint\limits_{c}^{d}f\left( \frac{a+b}{2}%
,y\right) dy\right]  \notag \\
&\leq &\frac{1}{\left( b-a\right) \left( d-c\right) }\dint\limits_{a}^{b}%
\dint\limits_{c}^{d}f(x,y)dxdy  \label{1.1} \\
&\leq &\frac{1}{4}\left[ \frac{1}{b-a}\dint\limits_{a}^{b}f\left( x,c\right)
dx+\frac{1}{b-a}\dint\limits_{a}^{b}f\left( x,d\right) dx\right.  \notag \\
&&\left. +\frac{1}{d-c}\dint\limits_{c}^{d}f\left( a,y\right) dy+\frac{1}{d-c%
}\dint\limits_{c}^{d}f\left( b,y\right) dy\right]  \notag \\
&\leq &\frac{f(a,c)+f(a,d)+f(b,c)+f(b,d)}{4}  \notag
\end{eqnarray}%
The above inequalities are sharp.
\end{theorem}

In \cite{SS}, Dragomir considered a mapping which closely connected with
above inequalities and established main properties of this mapping as
following:

Now, for a mapping $f:\Delta :=\left[ a,b\right] \times \left[ c,d\right]
\rightarrow 
\mathbb{R}
$ is convex on the co-ordinates on $\Delta ,$ we can define the mapping $H:%
\left[ 0,1\right] ^{2}\rightarrow 
\mathbb{R}
,$%
\begin{equation*}
H(t,s):=\frac{1}{\left( b-a\right) \left( d-c\right) }\dint\limits_{a}^{b}%
\dint\limits_{c}^{d}f\left( tx+(1-t)\frac{a+b}{2},sy+(1-s)\frac{c+d}{2}%
\right) dxdy
\end{equation*}

\begin{theorem}
Suppose that $f:\Delta \subset 
\mathbb{R}
^{2}\rightarrow 
\mathbb{R}
$ is convex on the co-ordinates on $\Delta =\left[ a,b\right] \times \left[
c,d\right] .$ Then:

(i) The mapping $H$ is convex on the co-ordinates on $\left[ 0,1\right]
^{2}. $

(ii) We have the bounds%
\begin{equation*}
\sup_{(t,s)\in \left[ 0,1\right] ^{2}}H(t,s)=\frac{1}{\left( b-a\right)
\left( d-c\right) }\dint\limits_{a}^{b}\dint\limits_{c}^{d}f\left(
x,y\right) dxdy=H(1,1)
\end{equation*}%
\begin{equation*}
\inf_{(t,s)\in \left[ 0,1\right] ^{2}}H(t,s)=f\left( \frac{a+b}{2},\frac{c+d%
}{2}\right) =H(0,0)
\end{equation*}

(iii) The mapping $H$ is monotonic nondecreasing on the co-ordinates.
\end{theorem}

For similar results see the papers \cite{LA}-\cite{HTS}.

In this paper, we defined convex functions on a rectanguler box and proved a
theorem similar to Theorem 1 on a rectanguler box, we also established a
mapping related to convex functions on a rectanguler box.

\section{Main Result}

Firstly we can define the convex function on a rectanguler box, as following:

\begin{definition}
Let us consider the rectangular box $G=[a,b]\times \lbrack c,d]\times
\lbrack e,f]$ in $%
\mathbb{R}
^{3}$. The mapping $f:G\rightarrow 
\mathbb{R}
$ is convex on $G$ if 
\begin{equation*}
f(tx+(1-t)u,ty+(1-t)v,tz+(1-t)w)\leq tf(x,y,z)+(1-t)f(u,v,w)
\end{equation*}%
holds for all $(x,y,z),$ $(u,v,w)\in G$ and $t\in \lbrack 0,1].$
\end{definition}

A function $f:G\rightarrow 
\mathbb{R}
$ is convex on $G$ is called co-ordinated convex on $G$ if the partial
mappings%
\begin{equation*}
f_{x}:[c,d]\times \lbrack e,f]\rightarrow 
\mathbb{R}
,\text{ \ \ \ }f_{x}(y,z)=f(x,y,z),\text{ }x\in \lbrack a,b]
\end{equation*}%
\begin{equation*}
f_{y}:[a,b]\times \lbrack e,f]\rightarrow 
\mathbb{R}
,\text{ \ \ \ }f_{y}(x,z)=f(x,y,z),\text{ }y\in \lbrack c,d]
\end{equation*}%
\begin{equation*}
f_{z}:[a,b]\times \lbrack c,d]\rightarrow 
\mathbb{R}
,\text{ \ \ \ }f_{z}(x,y)=f(x,y,z),\text{ }z\in \lbrack e,f]
\end{equation*}%
are convex for all $(y,z)\in \lbrack c,d]\times \lbrack e,f],$ $(x,z)\in
\lbrack a,b]\times \lbrack e,f],$ $(x,y)\in \lbrack a,b]\times \lbrack c,d].$

In order to prove our main theorem we need the following lemma.

\begin{lemma}
Every co-ordinated convex mapping $f:G\rightarrow 
\mathbb{R}
$ is convex for triple co-ordinates, where $G=[a,b]\times \lbrack c,d]\times
\lbrack e,f].$
\end{lemma}

\begin{proof}
Suppose that $f:G=[a,b]\times \lbrack c,d]\times \lbrack e,f]\rightarrow 
\mathbb{R}
$ is convex on $G.$ Recall the function 
\begin{equation*}
f_{x}:[c,d]\times \lbrack e,f]\rightarrow 
\mathbb{R}
,\text{ \ \ \ }f_{x}(y,z)=f(x,y,z),\text{ \ }x\in \lbrack a,b]
\end{equation*}

Then for $t\in \lbrack 0,1]$ and $(y_{1},z_{1}),$ $(y_{2},z_{2})\in \lbrack
c,d]\times \lbrack e,f],$ we have%
\begin{eqnarray*}
f_{x}(ty_{1}+(1-t)y_{2},tz_{1}+(1-t)z_{2})
&=&f(x,ty_{1}+(1-t)y_{2},tz_{1}+(1-t)z_{2}) \\
&=&f(tx+(1-t)x,ty_{1}+(1-t)y_{2},tz_{1}+(1-t)z_{2}) \\
&\leq &tf(x,y_{1},z_{1})+(1-t)f(x,y_{2},z_{2}) \\
&=&tf_{x}(y_{1},z_{1})+(1-t)f_{x}(y_{2},z_{2})
\end{eqnarray*}%
Therefore $f_{x}(y,z)=f(x,y,z)$ is convex on $[c,d]\times \lbrack e,f],$ for
all $x\in \lbrack a,b].$ The\ fact\ that\ $f_{y}:[a,b]\times \lbrack
e,f]\rightarrow 
\mathbb{R}
,$ $f_{y}(x,z)=f(x,y,z)$ is convex on $[a,b]\times \lbrack e,f]$ for all $%
y\in \lbrack c,d]$ and $f_{z}:[a,b]\times \lbrack c,d]\rightarrow 
\mathbb{R}
,$ $f_{z}(x,y)=f(x,y,z)$ is also convex on $[a,b]\times \lbrack c,d]$ for
all $z\in \lbrack e,f]$ goes likewise and we shall omit the details.
\end{proof}

\begin{theorem}
Suppose that $f:G=[a,b]\times \lbrack c,d]\times \lbrack e,f]\rightarrow 
\mathbb{R}
$ is convex on $G.$ Then one has the inequalities:%
\begin{eqnarray}
&&f\left( \frac{a+b}{2},\frac{c+d}{2},\frac{e+f}{2}\right)  \label{2.1} \\
&\leq &\frac{1}{3\left( b-a\right) }\dint\limits_{a}^{b}f\left( x,\frac{c+d}{%
2},\frac{e+f}{2}\right) dx+\frac{1}{3\left( d-c\right) }\dint%
\limits_{c}^{d}f\left( \frac{a+b}{2},y,\frac{e+f}{2}\right) dy  \notag \\
&&+\frac{1}{3\left( f-e\right) }\dint\limits_{e}^{f}f\left( \frac{a+b}{2},%
\frac{c+d}{2},z\right) dz  \notag \\
&\leq &\frac{1}{\left( b-a\right) \left( d-c\right) \left( f-e\right) }%
\diiint\nolimits_{G}f(x,y,z)dydzdx  \notag \\
&\leq &\frac{1}{6}\left[ \frac{1}{\left( b-a\right) \left( d-c\right) }%
\diint\nolimits_{\Delta _{1}}f\left( x,y,e\right) dydx+\frac{1}{\left(
b-a\right) \left( d-c\right) }\diint\nolimits_{\Delta _{1}}f\left(
x,y,f\right) dydx\right.  \notag \\
&&\left. +\frac{1}{\left( b-a\right) (f-e)}\diint\nolimits_{\Delta
_{2}}f(x,c,z)dzdx+\frac{1}{\left( b-a\right) (f-e)}\diint\nolimits_{\Delta
_{2}}f\left( x,d,z\right) dzdx\right.  \notag \\
&&\left. +\frac{1}{\left( d-c\right) (f-e)}\diint\nolimits_{\Delta
_{3}}f(a,y,z)dzdy+\frac{1}{\left( d-c\right) (f-e)}\diint\nolimits_{\Delta
_{3}}f\left( b,y,z\right) dzdy\right]  \notag \\
&\leq &\frac{f(a,c,e)+f(a,d,e)+f(b,c,e)+f(b,d,e)}{8}  \notag \\
&&+\frac{f(a,c,f)+f(a,d,f)+f(b,c,f)+f(b,d,f)}{8}  \notag
\end{eqnarray}%
where $\Delta _{1}=[a,b]\times \lbrack c,d],$ $\Delta _{2}=[a,b]\times
\lbrack e,f]$ and $\Delta _{3}=[c,d]\times \lbrack e,f].$
\end{theorem}

\begin{proof}
Since $f:G=[a,b]\times \lbrack c,d]\times \lbrack e,f]\rightarrow 
\mathbb{R}
$ is convex on $G,$ we can write the partial mappings%
\begin{equation*}
g_{x}:[c,d]\times \lbrack e,f]\rightarrow 
\mathbb{R}
,\text{ \ \ \ }g_{x}(y,z)=f(x,y,z)
\end{equation*}%
\begin{equation*}
g_{y}:[a,b]\times \lbrack e,f]\rightarrow 
\mathbb{R}
,\text{ \ \ \ }g_{y}(x,z)=f(x,y,z)
\end{equation*}%
\begin{equation*}
g_{z}:[a,b]\times \lbrack c,d]\rightarrow 
\mathbb{R}
,\text{ \ \ \ }g_{z}(x,y)=f(x,y,z)
\end{equation*}%
are convex for all $(y,z)\in \lbrack c,d]\times \lbrack e,f],$ $(x,z)\in
\lbrack a,b]\times \lbrack e,f],$ $(x,y)\in \lbrack a,b]\times \lbrack c,d].$
Then by the inequality of (\ref{1.1}), we have%
\begin{eqnarray*}
&&g_{x}\left( \frac{c+d}{2},\frac{e+f}{2}\right) \\
&\leq &\frac{1}{\left( d-c\right) \left( f-e\right) }\dint\limits_{c}^{d}%
\dint\limits_{e}^{f}g_{x}(y,z)dydz
\end{eqnarray*}%
That\ is%
\begin{eqnarray*}
&&f\left( x,\frac{c+d}{2},\frac{e+f}{2}\right) \\
&\leq &\frac{1}{\left( d-c\right) \left( f-e\right) }\dint\limits_{c}^{d}%
\dint\limits_{e}^{f}f(x,y,z)dydz \\
&=&\frac{1}{\left( d-c\right) \left( f-e\right) }\diint\nolimits_{\Delta
_{3}}f(x,y,z)dydz
\end{eqnarray*}%
Integrating this inequality on $\left[ a,b\right] $ respect to $x$ and
dividing both sides of inequalities $(b-a)$, we have%
\begin{eqnarray}
&&\frac{1}{b-a}\dint\limits_{a}^{b}f\left( x,\frac{c+d}{2},\frac{e+f}{2}%
\right) dx  \label{2.2} \\
&\leq &\frac{1}{\left( b-a\right) \left( d-c\right) \left( f-e\right) }%
\dint\limits_{a}^{b}\dint\limits_{c}^{d}\dint\limits_{e}^{f}f(x,y,z)dydzdx 
\notag \\
&=&\frac{1}{\left( b-a\right) \left( d-c\right) \left( f-e\right) }%
\diiint\nolimits_{G}f(x,y,z)dydzdx  \notag
\end{eqnarray}%
By\ a similar argument, we have%
\begin{eqnarray}
&&\frac{1}{d-c}\dint\limits_{c}^{d}f\left( \frac{a+b}{2},y,\frac{e+f}{2}%
\right) dy  \label{2.3} \\
&\leq &\frac{1}{\left( b-a\right) \left( d-c\right) \left( f-e\right) }%
\dint\limits_{a}^{b}\dint\limits_{c}^{d}\dint\limits_{e}^{f}f(x,y,z)dydzdx 
\notag \\
&=&\frac{1}{\left( b-a\right) \left( d-c\right) \left( f-e\right) }%
\diiint\nolimits_{G}f(x,y,z)dydzdx  \notag
\end{eqnarray}%
\begin{eqnarray}
&&\frac{1}{f-e}\dint\limits_{e}^{f}f\left( \frac{a+b}{2},\frac{c+d}{2}%
,z\right) dz  \label{2.4} \\
&\leq &\frac{1}{\left( b-a\right) \left( d-c\right) \left( f-e\right) }%
\dint\limits_{a}^{b}\dint\limits_{c}^{d}\dint\limits_{e}^{f}f(x,y,z)dydzdx 
\notag \\
&=&\frac{1}{\left( b-a\right) \left( d-c\right) \left( f-e\right) }%
\diiint\nolimits_{G}f(x,y,z)dydzdx  \notag
\end{eqnarray}%
By\ addition\ (\ref{2.2}), (\ref{2.3}) and (\ref{2.4}), we get the second
inequality of (\ref{2.1})%
\begin{eqnarray*}
&&\frac{1}{3\left( b-a\right) }\dint\limits_{a}^{b}f\left( x,\frac{c+d}{2},%
\frac{e+f}{2}\right) dx+\frac{1}{3\left( d-c\right) }\dint\limits_{c}^{d}f%
\left( \frac{a+b}{2},y,\frac{e+f}{2}\right) dy \\
&&+\frac{1}{3\left( f-e\right) }\dint\limits_{e}^{f}f\left( \frac{a+b}{2},%
\frac{c+d}{2},z\right) dz \\
&\leq &\frac{1}{\left( b-a\right) \left( d-c\right) \left( f-e\right) }%
\diiint\nolimits_{G}f(x,y,z)dydzdx
\end{eqnarray*}%
By the first inequality of (\ref{1.1}), we have%
\begin{eqnarray}
&&f\left( \frac{a+b}{2},\frac{c+d}{2},\frac{e+f}{2}\right)  \label{2.5} \\
&\leq &\frac{1}{2}\left[ \frac{1}{d-c}\dint\limits_{c}^{d}f\left( \frac{a+b}{%
2},y,\frac{e+f}{2}\right) dy+\frac{1}{f-e}\dint\limits_{e}^{f}f\left( \frac{%
a+b}{2},\frac{c+d}{2},z\right) dz\right]  \notag
\end{eqnarray}%
\begin{eqnarray}
&&f\left( \frac{a+b}{2},\frac{c+d}{2},\frac{e+f}{2}\right)  \label{2.6} \\
&\leq &\frac{1}{2}\left[ \frac{1}{b-a}\dint\limits_{a}^{b}f\left( x,\frac{c+d%
}{2},\frac{e+f}{2}\right) dx+\frac{1}{f-e}\dint\limits_{e}^{f}f\left( \frac{%
a+b}{2},\frac{c+d}{2},z\right) dz\right]  \notag
\end{eqnarray}%
\begin{eqnarray}
&&f\left( \frac{a+b}{2},\frac{c+d}{2},\frac{e+f}{2}\right)  \label{2.7} \\
&\leq &\frac{1}{2}\left[ \frac{1}{b-a}\dint\limits_{a}^{b}f\left( x,\frac{c+d%
}{2},\frac{e+f}{2}\right) dx+\frac{1}{d-c}\dint\limits_{c}^{d}f\left( \frac{%
a+b}{2},y,\frac{e+f}{2}\right) dy\right]  \notag
\end{eqnarray}%
Summing the inequalities (\ref{2.5})-(\ref{2.7}), we get the first
inequality of (\ref{2.1})%
\begin{eqnarray*}
&&f\left( \frac{a+b}{2},\frac{c+d}{2},\frac{e+f}{2}\right) \\
&\leq &\frac{1}{3}\left[ \frac{1}{d-c}\dint\limits_{c}^{d}f\left( \frac{a+b}{%
2},y,\frac{e+f}{2}\right) dy+\frac{1}{b-a}\dint\limits_{a}^{b}f\left( x,%
\frac{c+d}{2},\frac{e+f}{2}\right) dx\right. \\
&&\left. +\frac{1}{f-e}\dint\limits_{e}^{f}f\left( \frac{a+b}{2},\frac{c+d}{2%
},z\right) dz\right]
\end{eqnarray*}%
Finally, by the third and last inequalities of (\ref{1.1}), we have%
\begin{eqnarray}
&&\frac{1}{\left( b-a\right) \left( d-c\right) (f-e)}\dint\limits_{a}^{b}%
\dint\limits_{c}^{d}\dint\limits_{e}^{f}f(x,y,z)dxdydz  \label{2.8} \\
&\leq &\frac{1}{4}\left[ \frac{1}{\left( b-a\right) \left( d-c\right) }%
\dint\limits_{a}^{b}\dint\limits_{c}^{d}f\left( x,y,e\right) dydx+\frac{1}{%
\left( b-a\right) \left( d-c\right) }\dint\limits_{a}^{b}\dint%
\limits_{c}^{d}f\left( x,y,f\right) dydx\right.  \notag \\
&&\left. +\frac{1}{\left( b-a\right) (f-e)}\dint\limits_{a}^{b}\dint%
\limits_{e}^{f}f(x,c,z)dzdx+\frac{1}{\left( b-a\right) (f-e)}%
\dint\limits_{a}^{b}\dint\limits_{e}^{f}f\left( x,d,z\right) dzdx\right] 
\notag \\
&\leq &\frac{\dint\limits_{a}^{b}f(x,c,e)dx+\dint\limits_{a}^{b}f(x,c,f)dx+%
\dint\limits_{a}^{b}f(x,d,e)dx+\dint\limits_{a}^{b}f(x,d,f)dx}{4(b-a)} 
\notag
\end{eqnarray}%
\begin{eqnarray}
&&\frac{1}{\left( b-a\right) \left( d-c\right) (f-e)}\dint\limits_{a}^{b}%
\dint\limits_{c}^{d}\dint\limits_{e}^{f}f(x,y,z)dxdydz  \label{2.9} \\
&\leq &\frac{1}{4}\left[ \frac{1}{\left( b-a\right) \left( d-c\right) }%
\dint\limits_{c}^{d}\dint\limits_{a}^{b}f\left( x,y,e\right) dxdy+\frac{1}{%
\left( b-a\right) \left( d-c\right) }\dint\limits_{c}^{d}\dint%
\limits_{a}^{b}f\left( x,y,f\right) dxdy\right.  \notag \\
&&\left. +\frac{1}{\left( d-c\right) (f-e)}\dint\limits_{c}^{d}\dint%
\limits_{e}^{f}f(a,y,z)dzdy+\frac{1}{\left( d-c\right) (f-e)}%
\dint\limits_{c}^{d}\dint\limits_{e}^{f}f\left( b,y,z\right) dzdy\right] 
\notag \\
&\leq &\frac{\dint\limits_{c}^{d}f(a,y,e)dy+\dint\limits_{c}^{d}f(a,y,f)dy+%
\dint\limits_{c}^{d}f(b,y,e)dy+\dint\limits_{c}^{d}f(b,y,f)dy}{4(d-c)} 
\notag
\end{eqnarray}%
\begin{eqnarray}
&&\frac{1}{\left( b-a\right) \left( d-c\right) (f-e)}\dint\limits_{a}^{b}%
\dint\limits_{c}^{d}\dint\limits_{e}^{f}f(x,y,z)dxdydz  \label{2.10} \\
&\leq &\frac{1}{4}\left[ \frac{1}{\left( b-a\right) (f-e)}%
\dint\limits_{e}^{f}\dint\limits_{a}^{b}f\left( x,c,z\right) dxdz+\frac{1}{%
\left( b-a\right) (f-e)}\dint\limits_{e}^{f}\dint\limits_{a}^{b}f\left(
x,d,z\right) dxdz\right.  \notag \\
&&\left. +\frac{1}{\left( d-c\right) (f-e)}\dint\limits_{e}^{f}\dint%
\limits_{c}^{d}f(a,y,z)dydz+\frac{1}{\left( d-c\right) (f-e)}%
\dint\limits_{e}^{f}\dint\limits_{c}^{d}f\left( b,y,z\right) dydz\right] 
\notag \\
&\leq &\frac{\dint\limits_{e}^{f}f(a,c,z)dz+\dint\limits_{e}^{f}f(a,d,z)dz+%
\dint\limits_{e}^{f}f(b,c,z)dz+\dint\limits_{e}^{f}f(b,d,z)dz}{4(f-e)} 
\notag
\end{eqnarray}%
Summing the inequalities (\ref{2.8})-(\ref{2.10}) and computing the right
hand sides of above inequalities, we get the third and last inequalities of (%
\ref{2.1})%
\begin{eqnarray*}
&&\frac{1}{\left( b-a\right) \left( d-c\right) (f-e)}\dint\limits_{a}^{b}%
\dint\limits_{c}^{d}\dint\limits_{e}^{f}f(x,y,z)dxdydz \\
&\leq &\frac{1}{6}\left[ \frac{1}{\left( b-a\right) \left( d-c\right) }%
\dint\limits_{a}^{b}\dint\limits_{c}^{d}f\left( x,y,e\right) dydx+\frac{1}{%
\left( b-a\right) \left( d-c\right) }\dint\limits_{a}^{b}\dint%
\limits_{c}^{d}f\left( x,y,f\right) dydx\right. \\
&&\left. +\frac{1}{\left( b-a\right) (f-e)}\dint\limits_{a}^{b}\dint%
\limits_{e}^{f}f(x,c,z)dzdx+\frac{1}{\left( b-a\right) (f-e)}%
\dint\limits_{a}^{b}\dint\limits_{e}^{f}f\left( x,d,z\right) dzdx\right. \\
&&\left. +\frac{1}{\left( d-c\right) (f-e)}\dint\limits_{c}^{d}\dint%
\limits_{e}^{f}f(a,y,z)dzdy+\frac{1}{\left( d-c\right) (f-e)}%
\dint\limits_{c}^{d}\dint\limits_{e}^{f}f\left( b,y,z\right) dzdy\right] \\
&\leq &\frac{f(a,c,e)+f(a,d,e)+f(b,c,e)+f(b,d,e)}{4} \\
&&+\frac{f(a,c,f)+f(a,d,f)+f(b,c,f)+f(b,d,f)}{4}
\end{eqnarray*}%
\ which completes the proof.
\end{proof}

Now, for a mapping $f:G=[a,b]\times \lbrack c,d]\times \lbrack
e,f]\rightarrow 
\mathbb{R}
$ as above, we can define the mapping $H:\left[ 0,1\right] ^{3}\rightarrow 
\mathbb{R}
,$%
\begin{eqnarray*}
H(t,s,r) &:&=\frac{1}{\left( b-a\right) \left( d-c\right) \left( f-e\right) }
\\
&&\times \dint\limits_{a}^{b}\dint\limits_{c}^{d}\dint\limits_{e}^{f}f\left(
tx+(1-t)\frac{a+b}{2},sy+(1-s)\frac{c+d}{2},rz+(1-r)\frac{e+f}{2}\right)
dxdydz.
\end{eqnarray*}%
The following theorem contains the properties of this mapping.

\begin{theorem}
Suppose that $f:G=[a,b]\times \lbrack c,d]\times \lbrack e,f]\rightarrow 
\mathbb{R}
$ is convex on $G,$ then:

(i) The mapping $H$ is convex on the co-ordinates on $\left[ 0,1\right]
^{3}. $

(ii) We have the bounds%
\begin{equation*}
\sup_{(t,s,r)\in \left[ 0,1\right] ^{3}}H(t,s,r)=\frac{1}{\left( b-a\right)
\left( d-c\right) \left( f-e\right) }\dint \dint\limits_{G}\dint
f(x,y,z)dydzdx=H(1,1,1)
\end{equation*}%
\begin{equation*}
\inf_{(t,s,r)\in \left[ 0,1\right] ^{3}}H(t,s,r)=f\left( \frac{a+b}{2},\frac{%
c+d}{2},\frac{e+f}{2}\right) =H(0,0,0)
\end{equation*}

(iii) The mapping $H$ is monotonic nondecreasing on the co-ordinates.
\end{theorem}

\begin{proof}
(i) Fix $s,r\in \left[ 0,1\right] .$ For all $\alpha ,\beta ,\gamma \geq 0$
with $\alpha +\beta +\gamma =1$ and $t_{1},t_{2},t_{3}\in \left[ 0,1\right]
, $ we have%
\begin{eqnarray*}
&&H(\alpha t_{1}+\beta t_{2}+\gamma t_{3},s,r) \\
&=&\frac{1}{\left( b-a\right) \left( d-c\right) \left( f-e\right) } \\
&&\times \dint\limits_{a}^{b}\dint\limits_{c}^{d}\dint\limits_{e}^{f}f\left(
\left( \alpha t_{1}+\beta t_{2}+\gamma t_{3}\right) x+(1-\alpha t_{1}+\beta
t_{2}+\gamma t_{3})\frac{a+b}{2}\right. \\
&&\left. ,sy+(1-s)\frac{c+d}{2},rz+(1-r)\frac{e+f}{2}\right) dxdydz \\
&=&\frac{1}{\left( b-a\right) \left( d-c\right) \left( f-e\right) }%
\dint\limits_{a}^{b}\dint\limits_{c}^{d}\dint\limits_{e}^{f}f\left( \alpha
\left( t_{1}x+\left( 1-t_{1}\right) \frac{a+b}{2}\right) \right. \\
&&\beta \left( t_{2}x+\left( 1-t_{2}\right) \frac{a+b}{2}\right) +\gamma
\left( t_{3}x+\left( 1-t_{3}\right) \frac{a+b}{2}\right) , \\
&&\left. sy+(1-s)\frac{c+d}{2},rz+(1-r)\frac{e+f}{2}\right) dxdydz \\
&\leq &\alpha \frac{1}{\left( b-a\right) \left( d-c\right) \left( f-e\right) 
} \\
&&\times \dint\limits_{a}^{b}\dint\limits_{c}^{d}\dint\limits_{e}^{f}f\left(
t_{1}x+\left( 1-t_{1}\right) \frac{a+b}{2},sy+(1-s)\frac{c+d}{2},rz+(1-r)%
\frac{e+f}{2}\right) dxdydz \\
&&+\beta \frac{1}{\left( b-a\right) \left( d-c\right) \left( f-e\right) } \\
&&\times \dint\limits_{a}^{b}\dint\limits_{c}^{d}\dint\limits_{e}^{f}f\left(
t_{2}x+\left( 1-t_{2}\right) \frac{a+b}{2},sy+(1-s)\frac{c+d}{2},rz+(1-r)%
\frac{e+f}{2}\right) dxdydz \\
&&+\gamma \frac{1}{\left( b-a\right) \left( d-c\right) \left( f-e\right) } \\
&&\times \dint\limits_{a}^{b}\dint\limits_{c}^{d}\dint\limits_{e}^{f}f\left(
t_{3}x+\left( 1-t_{3}\right) \frac{a+b}{2},sy+(1-s)\frac{c+d}{2},rz+(1-r)%
\frac{e+f}{2}\right) dxdydz \\
&=&\alpha H(t_{1},s,r)+\beta H(t_{2},s,r)+\gamma H(t_{3},s,r).
\end{eqnarray*}%
By a similar way, one can see that%
\begin{equation*}
H(t,\alpha s_{1}+\beta s_{2}+\gamma s_{3},r)=\alpha H(t,s_{1},r)+\beta
H(t,s_{2},r)+\gamma H(t,s_{3},r)
\end{equation*}%
and%
\begin{equation*}
H(t,s,\alpha r_{1}+\beta r_{2}+\gamma r_{3})=\alpha H(t,s,r_{1})+\beta
H(t,s,r_{2})+\gamma H(t,s,r_{3})
\end{equation*}%
which completes the proof of (i).

(ii) Using convexity of $f$ on the triple co-ordinates on $G$ and Jensen's
integral inequality, we get;%
\begin{eqnarray*}
H(t,s,r) &\geq &\frac{1}{\left( b-a\right) }\dint\limits_{a}^{b}f\left(
tx+(1-t)\frac{a+b}{2},\frac{1}{\left( d-c\right) }\dint\limits_{c}^{d}\left[
sy+(1-s)\frac{c+d}{2}\right] dy\right. \\
&&\left. ,\frac{1}{\left( f-e\right) }\dint\limits_{e}^{f}\left[ rz+(1-r)%
\frac{e+f}{2}\right] dz\right) dx \\
&=&\frac{1}{\left( b-a\right) }\dint\limits_{a}^{b}f\left( tx+(1-t)\frac{a+b%
}{2},\frac{c+d}{2},\frac{e+f}{2}\right) dx \\
&\geq &f\left( \frac{1}{\left( b-a\right) }\dint\limits_{a}^{b}\left[
tx+(1-t)\frac{a+b}{2}\right] ,\frac{c+d}{2},\frac{e+f}{2}\right) dx \\
&=&f\left( \frac{a+b}{2},\frac{c+d}{2},\frac{e+f}{2}\right)
\end{eqnarray*}%
Since $H$ is convex on the co-ordinates, we have%
\begin{eqnarray*}
&&H(t,s,r) \\
&\leq &\frac{rst}{\left( b-a\right) \left( d-c\right) \left( f-e\right) }%
\dint\limits_{a}^{b}\dint\limits_{c}^{d}\dint\limits_{e}^{f}f(x,y,z)dxdydz+%
\frac{rs(1-t)}{\left( d-c\right) \left( f-e\right) }\dint\limits_{c}^{d}%
\dint\limits_{e}^{f}f(\frac{a+b}{2},y,z)dydz \\
&&+\frac{rt(1-s)}{\left( b-a\right) \left( f-e\right) }\dint\limits_{a}^{b}%
\dint\limits_{e}^{f}f(x,\frac{c+d}{2},z)dxdz+\frac{st(1-r)}{\left(
b-a\right) \left( d-c\right) }\dint\limits_{a}^{b}\dint\limits_{c}^{d}f(x,y,%
\frac{e+f}{2})dxdy \\
&&+\frac{t(1-s)(1-r)}{\left( b-a\right) }\dint\limits_{a}^{b}f(x,\frac{c+d}{2%
},\frac{e+f}{2})dx+\frac{s(1-t)(1-r)}{\left( d-c\right) }\dint%
\limits_{c}^{d}f(\frac{a+b}{2},y,\frac{e+f}{2})dy \\
&&+\frac{r(1-t)(1-s)}{\left( f-e\right) }\dint\limits_{e}^{f}f(\frac{a+b}{2},%
\frac{c+d}{2},z)dz+(1-r)(1-t)(1-s)f(\frac{a+b}{2},\frac{c+d}{2},\frac{e+f}{2}%
)
\end{eqnarray*}%
By using the inequalities of (\ref{2.1}), we obtain 
\begin{equation*}
H(t,s,r)\leq \frac{1}{\left( b-a\right) \left( d-c\right) \left( f-e\right) }%
\dint\limits_{a}^{b}\dint\limits_{c}^{d}\dint\limits_{e}^{f}f(x,y,z)dxdydz
\end{equation*}%
this completes the proof of (ii).

(iii) By using Hadamard's inequality, one can see that:%
\begin{equation*}
H(t,s,r)\geq H(0,s,r)
\end{equation*}%
for all $(t,s,r)\in \left[ 0,1\right] ^{3}.$ Let $0\leq t_{1}<t_{2}\leq 1.$
By the convexity of $H$, for all $(s,r)\in \lbrack 0,1]^{2},$ we have%
\begin{equation*}
\frac{H(t_{2},s,r)-H(t_{1},s,r)}{t_{2}-t_{1}}\geq \frac{H(t_{1},s,r)-H(0,s,r)%
}{t_{1}}\geq 0.
\end{equation*}%
Which completes the proof of (iii).
\end{proof}

\end{document}